\documentclass[12pt,a4paper]{article}
\usepackage{amsmath}

\usepackage{amsfonts}
\usepackage{amssymb}
\usepackage{amsmath}
\usepackage{amsthm}
\usepackage{amscd}
\usepackage{graphicx}
\usepackage{color}
\newcommand{\be}{\begin{equation}}
\newcommand{\ee}{\end{equation}}
\newcommand{\ef}[1]{\, #1}
\newcommand\rd{{\mathrm d}}

\newcommand\fois{\mathord{\cdot}}
\newcommand\dd{{\text{\textup{d}}}}
\newcommand\Imag{{\text{\textup{Im}}}}

\newcommand\caA{{\mathcal A}}

\newcommand\caJ{{\mathcal J}}

\newcommand\caP{{\mathcal P}}

\newcommand\gone{{ \mathchoice {1\mskip-4mu\mathrm{l} } {1\mskip-4mu\mathrm{l} }{1\mskip-4.5mu\mathrm{l} } {1\mskip-5mu\mathrm{l}} }}

\newcommand\gK{{\mathbb K}}

\newcommand\gC{{\mathbb C}}

\newcommand\gN{{\mathbb N}}

\newtheorem{Theorem}{Theorem}[section]
\newtheorem{theorem}[Theorem]{Theorem}

\newtheorem{proposition}[Theorem]{Proposition}

\newtheorem{lemma}[Theorem]{Lemma}

\newtheorem{conj}[Theorem]{Conjecture}

\newtheorem{remark}[Theorem]{Remark}

\newtheorem{definition}[Theorem]{Definition}

\newcommand{\tr}{\mathrm{Tr}} 

\newcommand{\beqa}{\begin{eqnarray}}
\newcommand{\eeqa}{\end{eqnarray}}

\newcommand{\omi}[1]{\buildrel { \buildrel{#1}\over{\vee} } \over .}

\usepackage[latin1]{inputenc}
\usepackage{subfigure}

\date{}
\author{A. de Goursac\thanks{{axelmg@melix.net}},
   A. Sportiello\thanks{andrea.sportiello@lipn.univ-paris13.fr}
  and A. Tanasa\thanks{adrian.tanasa@lipn.univ-paris13.fr}}
\title
{The Jacobian Conjecture, a Reduction of the Degree to the Quadratic Case}
\begin{document}
\maketitle
\begin{abstract}
The Jacobian Conjecture states that any locally invertible polynomial system in $\gC^n$ is globally invertible with polynomial inverse. C. W.  Bass {\it et. al.} (1982) proved a reduction theorem stating that the conjecture is true for any degree of the polynomial system if it is true in degree three. 
This degree reduction is obtained with the price of
increasing the dimension $n$.

We prove here a theorem concerning partial elimination of variables, which implies a reduction of the generic case to the quadratic one. The price to pay is the introduction of a supplementary parameter $0 \leq n' \leq n$,
parameter which represents the dimension of a linear subspace where some particular conditions on the system must hold.

We first give a purely algebraic proof of this reduction result 
and we then expose a distinct proof, in a Quantum Field Theoretical formulation,  using the {\it intermediate field method}.
\end{abstract}

\section{Introduction}
\label{sec:in}

The \emph{Jacobian conjecture} has been formulated in \cite{Keller},
as a strikingly simple and natural conjecture concerning the global
invertibility of polynomial systems. Later on, it also appeared
connected to questions in non-commutative algebra, in particular the
conjecture has been shown to be stably equivalent to the Dixmier
Conjecture (see \cite{Belov:2008}), which concerns endomorphisms of
the Weyl algebra. Despite several efforts, and various promising
partial results, it remains unsolved.  An introduction to the problem,
the context, and the state of the art up to 1982, can be found in the
paper
\cite{BassCW}, which provides both a clear review, and among the most
relevant advances on the problem.

The function $F: \gC^n \to \gC^n$ is said to be a 
\emph{polynomial system} if all the coordinate functions $F_j$'s are polynomials. Let us
call $\caP_n$ the set of such functions. For a function $F$, define
\be
\label{def.jac}
J_F(z) = 
\left(
\frac{\rd}{\rd z_i} F_j(z) \right)_{1 \leq i,j \leq n},
\ee
the corresponding Jacobian matrix. Then $\det J_F(z)$ is itself a
polynomial, and it is nowhere zero iff it is a non-zero constant.
Local invertibility in $y=F(z)$ is related to the invertibility of
$J_F(z)$ as a matrix, and thus to the vanishing of its
determinant. Depending from the space of functions under analysis,
local invertibility may or may not be sufficient to imply global
invertibility. The question here is what is the case, for the space
$\caP_n$ of polynomial systems over an algebraically-closed field of
characteristic zero (such as $\mathbb{C}$). For this reason, we define
the two subspaces of~$\caP_n$
\begin{definition}
\begin{align}
\caJ_n^{\rm lin}
&:= \{F\in \caP_n \,|\, \det J_F(z) = c \in \gC^{\times}\},\\
\caJ_n
&:=\{F\in \caP_n\, |\, F {\text{ is invertible}}\}.
\end{align}
\end{definition}
For the questions at hand, it will often be sufficient to analyse the
subset of $\caJ^{\rm lin}$ such that $\det J_F(z) = 1$.

More precisely w.r.t.\ what anticipated above, one can see, e.g., from
\cite[Thm.\ 2.1, pag.\ 294]{BassCW}, that $F\in \caJ_n^{\rm lin}$ is a
necessary condition for the invertibility of $F$, and, if $F$ is
invertible, the (set theoretic) inverse is automatically polynomial
and unique. The question is whether the condition on the Jacobian is
also sufficient, i.e.
\begin{conj}[Jacobian Conjecture, \cite{Keller}]
\label{conj.jc}
\be
\caJ_n^{\rm lin} = \caJ_n
\qquad
\forall n
\ef.
\ee
\end{conj}
Define the \emph{total degree} of $F$, $\deg(F)$, as $\max_j
\deg(F_j(z))$, and introduce the subspaces
\begin{align}
\caP_{n,d} &= \{ F \in \caP_n \;|\; \deg(F) \leq d \}
\ef;
\end{align}
and similarly for $\caJ$ and $\caJ^{\rm lin}$. We mention two positive
results on the conjecture. First, a theorem for the quadratic case
($d=2$), established first in \cite{Wang}, and then, through a much
simpler proof, in \cite{Oda} (see also \cite[Lemma 3.5]{Wright2} and
\cite[Thm.\ 2.4, pag.\ 298]{BassCW}).
\begin{theorem}[\cite{Wang}]
\label{thm.red2}
\be
\caJ_{n,2}^{\rm lin} = \caJ_{n,2}
\qquad
\forall n
\ef.
\ee
\end{theorem}
Then, a reduction theorem, from the general case to the cubic case, established by
Bass, Connell and Wright \cite[Sec.~II]{BassCW}).
\begin{theorem}[\cite{BassCW}]
\label{thm.red3}
\be
\caJ_{n,3}^{\rm lin} = \caJ_{n,3}
\quad
\forall n
\qquad
\Longrightarrow
\qquad
\caJ_{n}^{\rm lin} = \caJ_{n}
\quad
\forall n
\ef.
\ee
\end{theorem}
The proof of the above theorem involves manipulations under which the
dimension $n$ of the system is increased, thus this proof does
\emph{not} imply the corresponding statement without the ``$\forall
n$'' quantifier, i.e.\ that
$\caJ_{n,3}^{\rm lin} = \caJ_{n,3} \ \Rightarrow
\ \caJ_{n}^{\rm lin} = \caJ_{n} $.

%


\medskip

Our result is a reduction theorem, in the line of the one
above, that tries to `fill the gap' between the cases of degree two
and three.  However, for this, we have to introduce an adaptation of the
statement of the Jacobian conjecture with one more parameter.

For $n'\leq n$ and $F\in\caP_{n,d}$, we write $z=(z_1,z_2)$ and
$F=(F_1,F_2)$ to distinguish components in the two subspaces
$\gC^{n'}\times\gC^{n-n'} \equiv \gC^{n}$.
We set $R(z_2;z_1)$ to be equal to $F_2(z_1,z_2)$,
emphasizing that, in $R$, we consider $z_2$ as the variables in a
polynomial system, and $z_1$ as parameters. The invertibility of $R$, denoted by $R(\fois;z_1)\in\caJ_{n-n',d}$, for a fixed $z_1$, means that there exists a polynomial $R^{-1}$ with variables $y_2\in\gC^{n-n'}$, and depending on $z_1$, such that
\begin{equation*}
\forall z_2\in\gC^{n-n'},\quad R^{-1}(R(z_2;z_1);z_1)=z_2.
\end{equation*}
We are now ready to define the main objects involved in the reduction theorem of this paper.
\begin{definition}
\label{def-jjlin}
We define the
subspaces of $\caP_{n,d}$
\begin{align*}
\caJ_{n,d;n'}  :=&\{F\in\caP_{n,d}\,|\,R(\fois;z_1)\in\caJ_{n-n',d}\ \forall z_1\in\gC^{n'}\\
&{\text{ and }}
F^{-1}\text{~restricted to~}
\gC^{n'}\times\{0\}\text{ is in }\caP_{n'}
\}
\\
\caJ^{\rm lin}_{n,d;n'}  :=&\{F\in\caP_{n,d}\,|\,
R(\fois;z_1)\in\caJ_{n-n',d}\ \forall z_1\in\gC^{n'}\\
&\text{ and }(\det J_F)(z_1,R^{-1}(0,z_1))=c \in \mathbb{C}^{\times},\;\forall z_1\in\gC^{n'}\}
\end{align*}
\end{definition}
The first definition is a clear and natural generalisation of
$\caJ_{n,d}$, which corresponds to $n'=n$. The second one evaluates the Jacobian on a suitable
algebraic variety, candidate image of $\gC^{n'}\times\{0\}$ under 
$F^{-1}$
(the reason for this choice will become clear only in the following, see in particular Section \ref{sec-example}).
As a consequence,
we have
that $\caJ^{\rm lin}_{n,d;n'}\subseteq \caJ_{n,d;n'}$, similarly to
the original $n'=n$ case. 

Apparently, the most natural and intrinsic generalization would have been to
consider an arbitrary linear subspace of dimension $n-n'$, on which $z$
shall vanish, instead of the last $n-n'$ variables.  However, it
will be notationally simpler to restrict to our choice, and cause
no loss of generality for the problem at hand, which is clearly
$\mathrm{GL}(n,\mathbb{C})$-invariant.


Let us now state our reduction theorem:
\begin{Theorem}
\label{lem-main}
For $n\in\gN$ and $d\geq 3$, there exists an injective map 
$\Phi: \caP_{n,d}\to \caP_{n(n+1),d-1}$ 
satisfying 
\begin{align}
\Phi(\caJ^{\rm lin}_{n,d})
&\equiv
\caJ^{\rm lin}_{n(n+1), d-1;n}\cap \Imag(\Phi)
\ef;
&
\Phi(\caJ_{n,d})
&\equiv
\caJ_{n(n+1), d-1;n}\cap\Imag(\Phi)
\ef,
\end{align}
where $\Imag(\Phi)=\Phi(\caP_{n,d})$.
\end{Theorem}
Combining Theorem \ref{thm.red3} and the theorem above, the full
Jacobian Conjecture reduces to the question whether
\begin{equation*}
\caJ^{\rm lin}_{n(n+1),2;n}\cap\Imag(\Phi)=\caJ_{n(n+1),2;n}\cap\Imag(\Phi).
\end{equation*}
This question seems at a first stage as difficult as the original Jacobian conjecture. However, it involves only a quadratic degree, and this might simplify the resolution, in the light of Wang Theorem (Thm.~\ref{thm.red2}).


It is also natural, at this point, to formulate a stronger version of
the Jacobian Conjecture
\begin{conj}
\label{conj-main}
For all $n\geq n'\geq 0$, and all $d \geq 1$,
\be
\caJ^{\rm lin}_{n,d;n'}=\caJ_{n,d;n'}
\ef.
\ee
\end{conj}
As we have seen, the original Conjecture \ref{conj.jc} follows
from the cases\\
$(n,d;n') \in \{(m(m+1),2,m)\}_{m \geq 0}$ of the
above conjecture, restricted to $\Imag(\Phi)$.

\medskip

As already announced above, this paper proves Theorem \ref{lem-main}. We found this restriction of degree originally in the Quantum Field Theory formulation of the Jacobian conjecture, by applying the intermediate field method. So, in this paper, we first prove the theorem algebraically in section \ref{sec.algproof} and then, we expose the equivalent proof using combinatorial Quantum Field Theory (QFT) in subsection \ref{subsec-redthm}.
For
general references on QFT for combinatorists, we refer the interested
reader to \cite{malek-SLC} or \cite{adrian-SLC}, while a rederivation
of the QFT analogs of quantities pertinent here is done in subsection
\ref{subsec-eur}, at a heuristic level, and in subsection
\ref{subsec-inv}, more formally.

\medskip

For the sake of completeness, let us recall for the interested reader that various purely combinatorial approaches to the Jacobian Conjecture were given. Thus, the first paper of this series was the one of Zeilbelger \cite{zeil}, which proposes the Joyal method of combinatorial species as an appropriate tool to tackle the conjecture. His work has been followed by the one of Wright's \cite{wright}, which used trees to reformulate the conjecture and then by the one of Singer \cite{singer1}, which used rooted trees (see also \cite{w2} and \cite{singer2}).


\section{Algebraic proof of the reduction theorem}
\label{sec.algproof}

In this section we use algebraic methods to prove Theorem \ref{lem-main}. Let us first prove the following lemmas.

\begin{lemma}[Partial elimination, linearized version]
\label{lem.pelin}
Let $N=n_1+n_2$, and $S\in \caP_N$. Write $z=(z_1,z_2)$ for 
$z \in \gK^N = \gK^{n_1} \times \gK^{n_2}$ and so on.
Call $R(z_2; z_1)= S_2(z_1,z_2)$ the system in
$\caP_{n_2}$, where $z_1$ coordinates are intended as
parameters. Assume by hypothesis that $R(\fois; z_1) \in \caJ_{n_2}$ for
all $z_1 \in \gK^{n_1}$.
Define $H(z_1;y_2) = S_1(z_1,R^{-1}(y_2;z_1)) \in \caP_{n_1}$.
We have
\be
S \in \caJ_{N;n_1}^{\rm lin}
\qquad \textit{iff} \qquad
H(\fois ;0) \in \caJ_{n_1}^{\rm lin} 
\ef.
\ee
\end{lemma}
\begin{proof}
We actually prove that 
$$\det J_S(z_1,R^{-1}(y_2;z_1)) = (\det J_{R(\cdot;z_1)})\, \det J_{H(\cdot;y_2)}(z_1),$$ 
while
$\det J_{R(\cdot;z_1)} \in \gK^{\times}$ is fixed by hypothesis.

Let us start by calculating explicitly $\det J_S(z)$. Expressing
$S(z)$ in terms of $S_{1,2}$ and $z_{1,2}$ gives the block decomposition
\be
\frac{\rd}{\rd z} S(z)
=
\left(
\begin{array}{c|c}
\raisebox{-8pt}{\rule{0pt}{5pt}}%
\frac{\rd}{\rd z_1} S_1(z_1,z_2) &
  \frac{\rd}{\rd z_1} S_2(z_1,z_2) \\
\hline
\rule{0pt}{12pt}%
\frac{\rd}{\rd z_2} S_1(z_1,z_2) &
  \frac{\rd}{\rd z_2} S_2(z_1,z_2)
\end{array}
\right)
\ef.
\ee
We express the determinant of the matrix above by mean of the Schur
complement formula\footnote{For
  a block matrix $M=\begin{pmatrix} A & \!\!\!B \\ C & \!\!\!D \end{pmatrix}$, the
\emph{Schur complement formula} states 
$\det M = (\det D) \det (A - B D^{-1} C)$.}.
Recognize that $\frac{\rd}{\rd z_2} S_2(z_1,z_2)=
\frac{\rd}{\rd z_2} R(z_2;z_1) = J_{R(\cdot;z_1)}$.
Thus
\begin{multline*}
\det
\left(
\frac{\rd S(z)}{\rd z} 
\right)
=\\
(\det J_{R(\cdot;z_1)})
\;
\det
\left(
\frac{\rd S_1(z_1,z_2)}{\rd z_1} 
-
\frac{\rd S_2(z_1,z_2)}{\rd z_1} 
(J_{R(\cdot;z_1)})^{-1}
\frac{\rd S_1(z_1,z_2)}{\rd z_2} 
\right)
\ef.
\label{eq.543765b}
\end{multline*}
Recall that $z_2 = R^{-1}(y_2;z_1)$, however, if one wants to
substitute one for the other in the expression above, some care is
required in how to interpret derivatives w.r.t.\ $z_1$ and $z_2$.
We pass instead to the second evaluation
\be
\begin{split}
\det J_{H(\cdot;y_2)}(z_1) 
&=
\det
\left(
\frac{\rd H(z_1;y_2)}{\rd z_1} 
\right)
=
\det
\left(
\frac{\rd S_1(z_1;R^{-1}(y_2;z_1))}{\rd z_1} 
\right)
\ef.
\end{split}
\ee
We have to evaluate the total derivative w.r.t.\ the components of $z_1$.
In the rightmost matrix, we have two contributions, coming from
derivations acting on the first and second arguments, namely
\be
\begin{split}
\frac{\rd S_1(z_1;R^{-1}(y_2;z_1))}{\rd z_1} 
&=
\left.\left(
\frac{\rd S_1(z_1;R^{-1}(y_2;u_1))}{\rd z_1} 
+
\frac{\rd S_1(u_1;R^{-1}(y_2;z_1))}{\rd z_1} 
\right)\right|_{u_1=z_1}
\\
&=
\left(
\frac{\rd S_1(z_1;z_2)}{\rd z_1} 
+
\frac{\rd R^{-1}(y_2;z_1)}{\rd z_1} 
\frac{\rd S_1(z_1;z_2)}{\rd z_2} 
\right)
\end{split}
\label{eq.543765a}
\ee
and finally recognize that (here $I_n$ is the $n$-dimensional identity matrix)
\be
0 = 
\frac{\rd \,I_{n_2}}{\rd z_1} 
=
\frac{\rd R(R^{-1}(y_2;z_1);z_1)}{\rd z_1} 
=
\frac{\rd R(z_2;z_1)}{\rd z_1} 
+
\frac{\rd R^{-1}(y_2;z_1)}{\rd z_1} 
\frac{\rd R(z_2;z_1)}{\rd z_2} 
\ee
from which
\be
\frac{\rd R^{-1}(y_2;z_1)}{\rd z_1} 
=
-
\frac{\rd R(z_2;z_1)}{\rd z_1} 
\left(\frac{\rd R(z_2;z_1)}{\rd z_2} \right)^{-1}
=
-
\frac{\rd S_2(z_1,z_2)}{\rd z_1} 
(J_{R(\cdot;z_1)})^{-1}
\ef.
\label{eq.543765}
\ee
Substituting (\ref{eq.543765}) into (\ref{eq.543765a}), and
comparing to (\ref{eq.543765b}), leads to the conclusion.
\end{proof}

\begin{lemma}[Partial elimination, invertible version]
\label{lem.peinj}
Let $S$, $R$ and $H$ as in Lemma \ref{lem.pelin}. In particular, assume by hypothesis that $R(\fois; z_1) \in \caJ_{n_2}$ for
all $z_1 \in \gK^{n_1}$.
We have
\be
S(z) \in \caJ_{N;n_1} 
\qquad \textit{iff} \qquad
H(\fois;0) \in \caJ_{n_1} 
\ef.
\ee
\end{lemma}
\begin{proof}
We shall prove that $S^{-1}(\fois;0) \in \caP_{n_1}$ exists if and only if
$H^{-1}(\fois;0) \in \caP_{n_1}$ exists.

The direction ``$H \Rightarrow S$'' is obvious: just take $S^{-1}$ as
\begin{align}
(S^{-1})_1(y_1,y_2)
&=
H^{-1}(y_1;y_2)
\ef;
&
(S^{-1})_2(y_1,y_2)
&=
R^{-1}(y_2;H^{-1}(y_1;y_2))
\ef.
\end{align}
For the opposite direction, 
``$S \Rightarrow H$'', we start by assuming to have
\begin{align}
z_1 
&=
(S^{-1})_1(y_1,y_2)
\ef;
&
z_2 
&=
(S^{-1})_2(y_1,y_2)
\ef;
\end{align}
but we already know that $z_2 = R^{-1}(y_2;z_1)$. Thus, from the
unicity of the inverse, we obtain that
\be
\label{eq.654687676}
(S^{-1})_2(y_1,y_2)
=
R^{-1}(y_2; (S^{-1})_1(y_1,y_2) )
\ef.
\ee
From the definition of $H$ in terms of $S$, we get
\be
H(z_1;y_2) = S_1(z_1, R^{-1}(y_2;z_1) )
\ef.
\ee
Calculate
\be
\begin{split}
H \big( (S^{-1})_1(y_1,y_2 & );  y_2 \big) 
= 
S_1  \big( (S^{-1})_1(y_1,y_2),
  R^{-1}(y_2;(S^{-1})_1(y_1,y_2)) \big)
\\
&= 
S_1 \big( (S^{-1})_1(y_1,y_2),
  (S^{-1})_2(y_1,y_2) \big)
=
S_1(S^{-1}(y_1,y_2))
= y_1
\ef;
\end{split}
\ee
where we used Equation (\ref{eq.654687676}). From the unicity of the inverse
(when it exists), and its characterizing equation $H \big(
H^{-1}(y_1;y_2);y_2 \big) = y_1$, we can identify
\be
H^{-1}(y_1;y_2) = (S^{-1})_1(y_1,y_2)
\ee
and thus conclude.
 
\end{proof}

As announced above, we now prove 
Theorem \ref{lem-main}.


\begin{proof}
Let us outline the proof. From $F\in\caP_{n,d+1}$, we will construct a function $\tilde F=\Phi(F)\in\caP_{n(n+1),d}$ with $F(z^{(1)})= H(z^{(1)};0)$. Using Lemma \ref{lem.pelin}, if $F\in\caJ_{n,d+1}^{\rm lin}$, we then have that $\tilde F\in\caJ_{n(n+1),d,n}^{\rm lin}$, so
\begin{equation*}
\Phi(\caJ^{\rm lin}_{n,d+1})=\caJ^{\rm lin}_{n(n+1), d;n}\cap \Imag(\Phi).
\end{equation*}
By using Lemma \ref{lem.peinj} in a similar way, we also have $\Phi(\caJ_{n,d+1})=\caJ_{n(n+1), d;n}\cap \Imag(\Phi)$.  

We start from $F(z^{(1)})\in\caP_{n,d+1}$. Trivial arguments allow to establish that $F \in \caJ_n$ iff 
$F-F(0) \in \caJ_n$, i.e.\ we can drop the part of degree zero in
$F$ (see e.g.\ \cite[Prop.\ 1.1, pag.\ 303]{BassCW}).
Thus a generic $F\in\caP_{n,d+1}$ has the form
\begin{align}
F(z^{(1)})
&= \sum_{c=1}^{d+1}
F_c(z^{(1)})
\ef;
\end{align}
where $F_c$ is homogeneous of degree $c$.

From $F\in\caP_{n,d+1}$, we will construct a $\tilde F=\Phi(F)\in\caP_{n(n+1),d}$ such that we can identify $F(z^{(1)}) = H(z^{(1)};0)$, where $H(z^{(1)};y^{(2)})$ is associated to $\tilde F$ in the way this is done in Lemma \ref{lem.pelin}. We indeed use here notations similar to those of Lemma \ref{lem.pelin}, with $S=\tilde F$ and $N=n(n+1)$, except that,
as we use explicit component indices, we use upper-scripts for blocks
(e.g.\ $\tilde F^{(1)}(z^{(1)},z^{(2)})$, instead of $F_1(z_1,z_2)$, and
$z^{(1)}=\{z^{(1)}_i\}_{1 \leq i \leq n}$). Clearly we have $n_1=n$
and $n_2=n^2$.  As our construction is structured, we use double
indices for components in the second block,
i.e.\ $z^{(2)}=\{z^{(2)}_{ij}\}_{1 \leq i,j \leq n}$, instead of
$z^{(2)}=\{z^{(2)}_{\ell}\}_{1 \leq \ell \leq n^2}$.

Set now $\tilde F$ of degree at most $d$
and block dimensions $n$ and $n^2$, with explicit expression
\begin{align}
\tilde F^{(1)}_i (z^{(1)}, z^{(2)})
&=
\sum_j z^{(2)}_{ij} z^{(1)}_j 
\ef;
\\
\tilde F^{(2)}_{ij} (z^{(1)}, z^{(2)})
&=
z^{(2)}_{ij} -
\sum_c \frac{1}{c} \frac{\rd}{\rd z^{(1)}_j} \big(F_c\big)_i (z^{(1)})
\ef.
\end{align}
The complicated rightmost summand in 
$\tilde F^{(2)}_{ij}$ only depends on $z^{(1)}$, so that in fact
$R(z^{(2)};z^{(1)})$ is linear, and its invertibility is trivially
established, namely
\be
R^{-1}(y^{(2)};z^{(1)})
= y^{(2)}_{ij} +
\sum_c \frac{1}{c} \frac{\rd}{\rd z^{(1)}_j} \big(F_c\big)_i (z^{(1)})
\ef.
\ee
We only have to check that $H$, specialized to $y^{(2)}=0$, is equal to $F$, i.e.\ that
\be
F(z^{(1)})
= H(z^{(1)};0) 
:= \tilde F^{(1)}(z^{(1)},R^{-1}(0;z^{(1)}))
\ef.
\ee
Dropping the now useless superscripts
this reads
\be
\begin{split}
\big(\tilde F^{(1)}\big)_i (z,R^{-1}(0;z))
&=
\sum_j z_j \left( 0 + \sum_c \frac{1}{c} 
\frac{\rd}{\rd z_j} \big(F_c\big)_i (z)
\right)
\\
&=
\sum_c 
\frac{1}{c} 
\sum_j z_j 
\frac{\rd}{\rd z_j} \big(F_c\big)_i (z)
=
F_i (z)
\ef,
\end{split}
\ee
as was to be proven.\footnote{We used the obvious fact that if
  $A(x_1,\ldots,x_n)$ is a homogeneous polynomial of degree $d$, one
  has
\be
\frac{1}{d} \sum_{i=1}^n 
x_i \frac{\rd}{\rd x_i} A(x) = A(x)
\ef.
\ee
} 

\end{proof}









\section{QFT proof of the reduction theorem}

\subsection{Heuristic QFT proof}
\label{subsec-eur}

We consider in this section QFT arguments which are {\it heuristic},
as they involve rewritings of the involved algebraic quantities, in
terms of formal integrals, deformations of Gaussian integrals, that do
not generally converge. The reasonings implying (some of) the facts
that can be derived by these methods are illustrated in
\cite{malek-SLC}. Although, in our case, it would just be easier to
translate our procedure, step by step, into a purely algebraic one. We
did not perform this here, as we find that the QFT formalism provides
a notational shortcut and a useful visualization of the algebraic
derivation.




Let us start by briefly recalling what we shall call the
\emph{Abdesselam--Rivasseau model} (see \cite{Abdesselam0} for
details).  This model is a `zero-dimensional QFT model'.  Here, the
`dimension' $D$ refers to the fact that QFT models are stated in terms
of functional integrals, for field $\phi_i(x)$ depending on a discrete
index $i$ and a continuous coordinate $x \in \mathbb{R}^D$. Here we
are in the much simpler case $D=0$, i.e.\ we have only discrete
indices, this fact being in part responsible for the possibility of
producing rigorous proofs within this formalism (in such a situation,
some authors refer to a {\it combinatorial QFT}).  Note that $D$ shall
not be confused with the dimension $n$ of the linear system~$F(z)$.
We anticipate that our fields will be complex variables, in
holomorphic basis, and the associated integrals will be on $\gC^n$
(i.e., with measure ${\rm d} \phi \,{\rm d} \phi^\dag$).



Now, let $n,d\geq 1$, and let $F \in \caP_{n,d}$.
Invertibility of $F_1(z)$ is equivalent to the invertibility of
$F_2(z):=F_1(Rz+u)$, for $R \in \mathrm{GL}(n,\mathbb{C})$ and $u \in
\mathbb{C}^n$, and $F_1$, $F_2$ have the same degree, thus w.l.o.g.\ 
we can assume that $F(z)=z+\mathcal{O}(z^2)$.
In such a case the coordinate functions of $F$ can be written as
\begin{equation*}
F_i(z)=z_i-\sum_{k=2}^d\sum_{j_1,...,j_k=1}^n
w^{(k)}_{i,j_1...j_k}z_{j_1}...z_{j_k}=: z_i-\sum_{k=2}^d W_i^{(k)}(z)
\ef,
\end{equation*}
for $i\leq n$ and $w^{(k)}_{i,j_1...j_k}$ some
coefficients.\footnote{Although only the symmetrized quantities
  $\sum_{\sigma \in \mathfrak{S}_k}
  w^{(k)}_{i,j_{\sigma(1)}...j_{\sigma(k)}}$ contribute, it is
  convenient to keep this redundant notation.}  We introduce the
inhomogeneous extension of the QFT model of \cite{Abdesselam0}.
\begin{equation*}
Z(J,K)=\int_{\gC^n}\dd\varphi\dd\varphi^\dag e^{-\varphi^\dag\varphi+\varphi^\dag \sum_{k=2}^d W^{(k)}(\varphi)+J^\dag\varphi+\varphi^\dag K},
\end{equation*}
where $J$, $K$ are vectors in $\gC^n$. The full expression is called
\emph{partition function}, the expression in the exponential is called \emph{action},
the coefficients $w$ are called
\emph{coupling constants}, while $J$ and $K$ are called \emph{external sources}.
When the coupling constants are set to zero, the integral
is calculated 
by Gaussian integration:
\begin{equation}
\int_{\gC^n}\dd\varphi\dd\varphi^\dag e^{-\varphi^\dag\varphi+J^\dag\varphi+\varphi^\dag K}=e^{J^\dag K}.
\label{eq-gauss}
\end{equation}
We can then express the unique formal inverse $G$ of $F$. Indeed, for $H_i$ an analytic function and $u\in\gC^n$,
\begin{multline*}
\int\dd\varphi\dd\varphi^\dag H_i(\varphi)e^{-\varphi^\dag F(\varphi)+\varphi^\dag u}= \int\dd\tilde\varphi\dd\varphi^\dag H_i(G(\tilde\varphi+u))e^{-\varphi^\dag \tilde\varphi}\det(\partial G(\tilde\varphi+u))\\
=\int\dd\tilde\varphi H_i(G(\tilde\varphi+u))\delta(\tilde\varphi)\det(\partial G(\tilde\varphi+u))
=H_i(G(u))\det(\partial G(u)),
\end{multline*}
with the change of variables: $\tilde\varphi=F(\varphi)-u$.  Taking
the ratio of such expressions, for $H_i(z)=z_i$ at numerator, and
$H_i(z)=1$ at denominator, we obtain that the formal inverse
corresponds to the one-point (outgoing) correlation function:
\begin{equation}
G_i(u)=\frac{\int_{\gC^n}\dd\varphi\dd\varphi^\dag \varphi_ie^{-\varphi^\dag\varphi+\varphi^\dag \sum_{k=2}^d W^{(k)}(z)+\varphi^\dag u}}{\int_{\gC^n}\dd\varphi\dd\varphi^\dag e^{-\varphi^\dag\varphi+\varphi^\dag \sum_{k=2}^d W^{(k)}(z)+\varphi^\dag u}}\label{eq-1point}
\end{equation}
Moreover, the partition function coincides with the inverse of the Jacobian:
\begin{equation*}
Z(0,u)=\det(\partial G(u))= JG(u)=\frac{1}{JF(G(u))}.
\end{equation*}
The sets of polynomial functions involved in the Jacobian Conjecture
can be rephrased in this framework:
\begin{align*}
\caJ^{\rm lin}_{n,d}
&=
\{F\in\caP_{n,d}\,|\, Z(0,u)=1\,\forall u\in\gC^n\},
\\
\caJ_{n,d}
&=
\{F\in\caP_{n,d}\,|\,G_i(u)\text{ given by \eqref{eq-1point} is in
}\caP_{n}\}.
\end{align*}
\medskip

Let us now introduce the {\it intermediate field method} to reduce the
degree $d$ of $F$. We will thus add $n^2$ ``intermediate fields''
$\sigma$ to the model. Indeed, we have, from the general formula
(\ref{eq-gauss}) of Gaussian integration,
\begin{multline}
\label{eq-identfond}
e^{( \varphi^\dag_{i}\varphi_j )
\big(
\sum_{j_2,...,j_d=1}^n
w^{(d)}_{i,j,j_2...j_d}\varphi_{j_2}...\varphi_{j_d} \big) }
\\
=
\int_{\gC^{n^2}}\dd\sigma_{i,j}\dd\sigma^\dag_{i,j} 
e^{-\sigma_{i,j}^\dag\sigma_{i,j}
+\sigma^\dag_{i,j}
\big(
\sum_{j_2,...,j_d=1}^n
w^{(d)}_{i,j,j_2...j_d}\varphi_{j_2}...\varphi_{j_d}
\big)
+(
\varphi^\dag_{i}\varphi_j
)
\sigma_{i,j}}
\end{multline}
We now use the identity \eqref{eq-identfond}, for each pair $(i,j)$,
in the partition function of the model with $n$ dimensions and
degree $d$, in order to to re-express the monomials of degree $d$ in
the fields $\varphi$. This leads to:
\begin{multline*}
Z(J,K)=\int_{\gC^n}\dd\varphi\dd\varphi^\dag\int_{\gC^{n^2}}\dd\sigma\dd\sigma^\dag e^{-\varphi^\dag\varphi+\varphi^\dag \sum_{k=2}^{d-1} W^{(k)}(\varphi)+J^\dag\varphi+\varphi^\dag K}\\ e^{\sum_{i,j=1}^n\Big(-\sigma_{i,j}^\dag\sigma_{i,j}+\sigma^\dag_{i,j}\sum_{j_2,...,j_d=1}^n w^{(d)}_{i,j,j_2...j_d}\varphi_{j_2}...\varphi_{j_d}+\varphi^\dag_{i}\varphi_j\sigma_{i,j} \Big)}.
\end{multline*}

We define the new vector $\phi$ of $\gC^{n+n^2}$ by
$\phi=(\varphi_1,\ldots,\varphi_n,
\sigma_{1,1}, \ldots, \sigma_{1,n},$ 
$\cdots,
\sigma_{n,1}, \ldots, \sigma_{n,n})$.
We further define the interaction coupling constants $\tilde w$ as:
\begin{itemize}
\item for $k=d-1$, we set $\tilde w^{(d-1)}_{i,j,j_2...j_d}:=w^{(d-1)}_{i,j,j_2...j_d}$ and $\tilde w^{(d-1)}_{i\fois n+j,j_2...j_d}=w^{(d)}_{i,j,j_2...j_d}$ with $i,j,j_2,...j_n\leq n$
\item for $k\in\{3,...,d-2\}$, we set $\tilde w^{(k)}_{i,j,j_2...j_k}:=w^{(k)}_{i,j,j_2...j_k}$ with $i,j,j_2,...j_n\leq n$
\item for $k=2$, we set $\tilde w^{(2)}_{i,j,j_2}:=w^{(2)}_{i,j,j_2}$ and $\tilde w^{(2)}_{i,j,i\fois n+j}=1$ with $i,j,j_2\leq n$.
\end{itemize}
The remaining coefficients of $\tilde w$ are set to 0. 

In the same way, the external sources are defined to be
$\tilde J = (J,0)$ and $\tilde K = (K,0)$, where, of course, the
number of extra vanishing coordinates is $n^2$.
It is important to note that these external sources have fewer degrees
of freedom than coordinates ($n$ vs.\ $n(n+1)$).  We also remark that,
for generic $d$, in order to have a relation adapted to induction, it
is crucial to consider an inhomogeneous model, since the intermediate
field method originates terms of degrees $d-1$ and $3$.

One now has
\begin{equation*}
Z(J,K)=\int_{\gC^{n+n^2}}\dd\phi\dd\phi^\dag e^{-\phi^\dag\phi+\phi^\dag \sum_{k=2}^{d-1}\tilde W^{(k)}(\phi)+\tilde J^\dag\phi+\phi^\dag \tilde K}
\end{equation*}
and
\begin{equation*}
G_i(u)=\frac{\int_{\gC^{n+n^2}}\dd\phi\dd\phi^\dag \phi_ie^{-\phi^\dag\phi+\phi^\dag \sum_{k=2}^{d-1}\tilde W^{(k)}(\phi)+\phi^\dag \tilde u}}{\int_{\gC^{n+n^2}}\dd\phi\dd\phi^\dag e^{-\phi^\dag\phi+\phi^\dag \sum_{k=2}^{d-1}\tilde W^{(k)}(\phi)+\phi^\dag \tilde u}},
\end{equation*}
for $i\in\{1,\dots,n\}$. 

We have thus showed in a heuristic way that the partition function
(resp.\ the one-point correlation function) of the model with dimension
$n\in\gN$ and degree $d\in\gN\setminus\{1,2\}$ is equal to the
partition function (resp.\ the $n$ first coordinates of the one-point
correlation function) of the model with dimension $n(n+1)$ and degree
$d-1$, up to a redefinition of the coupling constant $w\mapsto\tilde
w$ and a trivial redefinition of the external sources. Since the
partition function corresponds to the inverse of the Jacobian
(resp.\ the one-point correlation function corresponds to the formal
inverse), this gives, as announced above, a heuristic proof of
Theorem \ref{lem-main}.

\subsection{Formal inverse in QFT}
\label{subsec-inv}

In this section we adopt notations as above, but we proceed at a more
formal level. In particular, the coefficients $w$ are considered as
formal indeterminates in (multi-dimensional) power series. In order to
project more easily to the simpler context of univariate power series,
we introduce a further, redundant, indeterminate $\theta$, by
replacing each coefficient $w^{(k)}_{i,j_1...j_k}$ by
$\theta^{k-1}w^{(k)}_{i,j_1...j_k}$.  We denote by $\gC[[\theta]]$ the
ring of formal power series in $\theta$.  The exponent of $\theta$ in
the $w$'s measures the ``spin'' of the associated monomial, i.e., the
action is invariant under the transformation
$\phi_j \to \phi_j e^{i \omega}$, $\phi^\dag_j \to 
\phi^\dag_j e^{-i \omega}$, $\theta \to \theta e^{-i \omega}$.

The polynomial function $F$ now is extended naturally to a function from
$\gC[[\theta]]^n$ to itself, although we are ultimately interested on
invertibility on $\gC^n$.
The integrals of the previous subsection 
are now well defined as a formal expansion in $\theta$
\begin{equation}
Z(J,K):=
\int_{\gC^n}\dd\varphi\dd\varphi^\dag
e^{-\varphi^\dag\varphi+J^\dag\varphi+\varphi^\dag K}
\sum_{r=0}^\infty \frac{1}{r!}\Big(\varphi^\dag
\sum_{k=2}^d(\theta^{k-1} W^{(k)}(\varphi)\Big)^r, 
\label{eq-partfunc}
\end{equation}
where, as in the previous section, $J$ and $K$ can be considered as
vectors in $\gC^n$ (as they enter the quadratic part of the action,
there is no need of promoting them to formal indeterminates).

Any term of this power series, i.e.\ $[\theta^r] Z(J,K)$ for $p \in
\mathbb{N}$, can be calculated as a finite linear combination of terms
of the following form, with $r=q-p$
\beqa
\int_{\gC^n}\dd\varphi\dd\varphi^\dag
e^{-\varphi^\dag\varphi+J^\dag\varphi+\varphi^\dag K} 
\varphi_{i_1}^\dag\ldots\varphi_{i_p}^\dag \varphi_{j_1}\ldots\varphi_{j_q}
\label{eq-gen1}
\eeqa
which, of course, is also given by
\beqa
\frac{\partial^p}{
\partial K_{i_1} \ldots\partial K_{i_p} 
}
\frac{\partial^q}{
\partial J^\dag_{j_1}\ldots\partial J^\dag_{j_q}
}
\int_{\gC^n}\dd\varphi\dd\varphi^\dag e^{-\varphi^\dag\varphi+J^\dag\varphi+\varphi^\dag K} 
\label{eq-gen1bis}
\eeqa
An analysis of such an expression, in light of \eqref{eq-gauss}, leads
to the \emph{Wick Theorem} (see for example textbooks such as \cite{book-cm}):
the integral in \eqref{eq-gen1} is equal to the sum over all
possible substitutions, in the monomial
$\varphi_{i_1}^\dag\dots\varphi_{i_p}^\dag
\varphi_{j_1}\dots\varphi_{j_q}$ of the patterns $\varphi_i^\dag \to
J_i^\dag$, $\varphi_j \to K_j$, and $\varphi_i^\dag\varphi_j \to
\delta_{ij}$, up to have no $\varphi$'s and $\varphi^\dag$'s left.

One can then associate graphical representations 
$\Gamma$
to these expressions,
going under the name of \emph{Feynman graphs}.
For the model analyzed here, the set of graphs and their associated
weights are obtained through the following rules:
\begin{itemize}
\item vertices in the graph have indices $i \in \{1,\ldots,n\}$
  attached to the incident edges;
\item a term $\delta_{ij}$
is represented as a directed edge, with index $i$ at its endpoints
\item a term $J_i^\dag$ is represented as a vertex of in-degree 1 and
  out-degree 0, incident to an edge of index $i$;
\item similarly, a term $K_j$ is represented as a vertex of in-degree 0 and
  out-degree 1, incident to an edge of index $i$;
\item a weight $\theta^{k-1}w^{(k)}_{i,j_1...j_k}$ is associated to a
  vertex with out-degree $1$ (and index $i$), and in-degree $k$ (and
  indices $\{j_1,...,j_k\}$). The incident edges have a cyclic ordering;
\item a symmetry factor $1/|\mathrm{Aut}(\Gamma)|$ appears overall, as
  a combination of the $1/r!$ factor and of multiple counting for the
  same diagram in the expansion of (\ref{eq-partfunc}).
\end{itemize}
Note that, in general, several but finitely many graphs contribute to
a given expression associated to a monomial
$\varphi_{i_1}^\dag\ldots\varphi_{i_p}^\dag \varphi_{j_1}\ldots\varphi_{j_q}$.


For the problem at hand here, both when evaluating the Jacobian and
the formal inverse of a component, we can restrict to the case
$J=0$. In order to conform to notations in the literature, we shall
also rename $K_j$'s into $u_j$'s (indeed, when $J=0$, the source $K$
induces a formal translation of the components $\phi$'s, without
translating the $\phi^\dag$'s).
It is easy to see that, in a theory with such a constrained set of
vertex out-degrees,
all contributing Feynman graphs 
with no $J$-leaves
contain exactly
one cycle\footnote{Cycles are commonly called ``loops'' in QFT.}
per connected component (see Fig.\ \ref{loop}), while connected graphs
with a unique $J$-leaf
correspond to directed trees rooted at this leaf.
\begin{figure}[!tb]
  \centering
\includegraphics[scale=1]{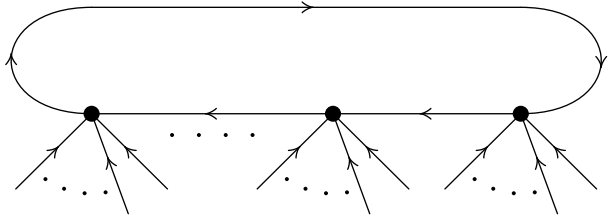}
  \caption[sigma]{Diagrammatic representation of a Feynman graph with
    no external outgoing edges, thus having one loop.}
  \label{loop}
\end{figure}



We can now state and prove the expression for the formal inverse,
which in fact coincides with the heuristically-derived
\eqref{eq-1point}, as well as the expression of the partition function
\eqref{eq-partfunc}. Both have been already derived
in~\cite{Abdesselam0}.

\begin{theorem}
\label{thm-inv}
Define $\caA_{i}(T)(u)$ the amplitude of a tree $T$ with exactly one
outgoing edge, of index $i$.
The formal inverse of the function $F$ is the function $G$ with coordinates
\begin{equation}
G_i(u)=\sum_{r=0}^\infty \frac{1}{r!}\sum_{T: |T|=r}\caA_i(T)(u),
\end{equation}
where $T$ 
denotes a tree with one outgoing edge, and $|T|$ is the number of
vertices in $T$.
\end{theorem}
Note an abuse of notation here, as $i$ is not an index pertinent to
the function $\caA$, but to the variable $T$. We can equivalently
think that $T$ has its root unlabeled (and the vertex adjacent to the
root comes with no weight), and the function $\caA_i$ completes the
weight of $T$ by including the appropriate factor $w^{(k)}_{i,j_1...j_k}$.

\begin{proof}
The proof is a straightforward adaptation of \cite{Abdesselam0} to the
inhomogeneous case. If the tree $T$ has at least one vertex, its
Feynman integral can be re-expressed as
\begin{equation*}
\caA_i(T)=\sum_{k=2}^d\theta^{k-1}\sum_{i_1,\dots,i_k=1}^n w^{(k)}_{i,i_1...i_k}\caA_{i_1}(T_1)\dots\caA_{i_k}(T_k),
\end{equation*}
where we have denoted by $T_1,\dots,T_k$ the subtrees with one
outgoing external edge which is adjacent in $T$ to the vertex linked
to the external edge $i$.  If $r_1,\ldots,r_k$ denote the sizes of
these subtrees, we must have $r_1+\ldots+r_k=r-1$.  In this sum, one
has to take into consideration the aforementioned symmetry
factors. However, the automorphism group of a rooted tree is trivially
expressed in terms of those of its branches, and the labels of the
branches. We have a further multinomial factor $\frac{r!}{1!\, r_1!\dots
  r_k!}$ for the total ordering of the vertices.  Since $G_i(u)$ is
the sum of Feynman integrals of graphs with one outgoing edge of index
$i$, we have
\begin{multline}
G_i(u)=u_i+\sum_{r=0}^\infty \frac{1}{r!}\sum_{k=2}^d\theta^{k-1}\sum_{i_1,\dots,i_k=1}^n w^{(k)}_{i,i_1...i_k}\sum_{T_1,\dots,T_k}\frac{r!}{r_1!\dots r_k!}\caA_{i_1}(T_1)\dots\caA_{i_k}(T_k)\\
=u_i+\sum_{k=2}^d\theta^{k-1}\sum_{i_1,\dots,i_k=1}^n w^{(k)}_{i,i_1...i_k}G_{i_1}(u)\dots G_{i_k}(u).\label{eq-invgraph}
\end{multline}
This implies that $F_i(G(u))=u_i$, and $G$ is the formal inverse of
$F$ on $\gC[[\theta]]^n$.
\end{proof}

\begin{remark}
From the aforementioned homogeneity of $\theta$, we also
have $G_i(u, \theta)= \lambda G_i(\lambda u, \lambda^{-1} \theta)$ for
$\lambda \in \mathbb{C}[[\theta]]^{\times}$.
This essentially allows to eliminate $\theta$, and, adapting an
argument of \cite{Abdesselam0}, show that $G_i$ is actually
analytic on a certain domain of convergence in the variables $u_j$.
\end{remark}


One then has:

\begin{proposition}
\label{prop-partfunc}
The partition function \eqref{eq-partfunc} of the above theory is given by
\begin{equation*}
Z(0,u)=\frac{1}{\det[J(F)(
G(u))]}.
\end{equation*}
\end{proposition}
\begin{proof}
The partition function $Z(0,u)$ is the sum of the weights of
the Feynman graphs without external edges.
Standard combinatorial arguments imply
that the logarithm $\ln(Z(0,u))$ is the analogous sum,
restricted to connected Feynman graphs.
In the Abdesselam-Rivasseau QFT model, these graphs have exactly one
loop. The overall contribution of graphs without external edges and
with one loop of length $r$ gives
\begin{multline*}
\frac 1r\sum_{k_1,..,k_r=1}^d\theta^{k_1+..+k_r-r}\sum_{\ell_1=1}^{k_1}\dots\sum_{\ell_r=1}^{k_r}\,\, \sum_{j_{(1,1)},...,j_{(r,k_r)}=1}^n w^{(k_1)}_{j_{(r,\ell_r)},j_{(1,1)},j_{(1,2)}...j_{(1,k_1)}}\\
G_{j_{(1,1)}}(u)G_{j_{(1,2)}}(u)..\omi{j_{(1,\ell_1)}}..G_{j_{(1,k_1)}}(u)
\;
w^{(k_2)}_{j_{(1,\ell_1)},j_{(2,1)}...j_{(2,k_2)}}G_{j_{(2,1)}}(u)..\omi{j_{(2,\ell_2)}}..\\
..G_{j_{(2,k_2)}}(u) \dots
w^{(k_r)}_{j_{(r-1,\ell_{r-1})},j_{(r,1)}...j_{(r,k_r)}} G_{j_{(r,1)}}(u)..\omi{j_{(r,\ell_r)}}..G_{j_{(r,k_r)}}(u),
\end{multline*}
where $..\omi{j_{(1,\ell_1)}}..$ means that the factor
$G_{j_{(1,\ell_1)}}(u)$ has been omitted in the product. The factor
$\frac 1r$ above is the residual symmetry factor, naturally
corresponding to the choice of the first vertex in an oriented
cycle. Note that
\begin{multline*}
\sum_{\ell_1=1}^{k_1}\,\, \sum_{j_{(1,1)},...,j_{(1,k_1)}=1}^n w^{(k_1)}_{j_{(r,\ell_r)},j_{(1,1)},j_{(1,2)}...j_{(1,k_1)}} G_{j_{(1,1)}}(u)G_{j_{(1,2)}}(u)..\omi{j_{(1,\ell_1)}}..G_{j_{(1,k_1)}}(u)\\
= \sum_{i,j=1}^n
\left.
\frac{\partial W_i^{(k_1)}(z)}{\partial z_j}\right|_{z=G(u)}
,
\end{multline*}
so that, omitting the now trivial factor $\theta$, the whole
contribution corresponds to:
\begin{equation*}
\frac{1}{r}\sum_{i_1,...,i_r=1}^n
\Big(\delta_{i_1i_2}-J(F)_{i_1i_2}(G(u))\Big)
\Big(\delta_{i_2i_3}-J(F)_{i_2i_3}(G(u))\Big)
\dots
\Big(\delta_{i_r i_1}-J(F)_{i_r i_1}(G(u))\Big).
\end{equation*}
One then has
\begin{equation}
\ln(Z(0,u))=\sum_{r=1}^\infty \frac{1}{r} 
\tr\Big(((\gone-J(F)(G(u)))^r\Big)
=-\tr\Big(\ln(J(F)(G(u)))\Big),
\end{equation}
which concludes the proof.
\end{proof}

\subsection{Proof of the theorem}
\label{subsec-redthm}

We give in this section the combinatorial QFT proof of our main
result, Theorem \ref{lem-main}.

\begin{proof}
Consider a directed tree $T$, constructed from the Feynman rules
defined in section \ref{subsec-inv}, in dimension $n$ (i.e.\ with edge
indices in $\{1,\ldots,n\}$), and degree $d$ (i.e.\ with vertices of
in-degree at most $d$).  From Theorem \ref{thm-inv}, we know that
$\caA_i(T)(u)$ is used to compute the formal inverses $G_i(u)$.

Let us now define the Feynman rules for the model in dimension
$n(n+1)$, obtained using the intermediate field method described in
Section \ref{subsec-eur}.  For $i,j\in\{1,\dots,n(n+1)\}$, one has:
\begin{itemize}
\item propagators, i.e. 
directed edges with index $i$
correspond to the term $\delta_{ij}$ (obtained by the
$\phi_i^\dag\phi_j$ substitution in Wick Theorem);
\item leaves with in-degree 1 correspond to the term $\tilde u_j$,
  and in particular, as $\tilde u_j = 0$ for $j>n$, contributing
  diagrams have all leaf-indices in the range $\{1,\ldots,n\}$;
\item vertices of coordination $k+1$, for $k\in\{1,\dots,d-1\}$, with one outgoing edge of index $i$, and $k$ incoming edges of indices $j_1,\dots,j_k$, correspond to the term $\theta^{k-1}\tilde w^{(k)}_{i,j_1...j_k}$. Indices of the vertices are summed on.
\end{itemize}
For a graphical representation of the intermediate field method leading to the model in dimension $n(n+1)$, see Figure~\ref{fig-sigma}.
\begin{figure}[!tb]
  \centering
\setlength{\unitlength}{50pt}
\begin{picture}(5.1,1.9)
\put(0,0){\includegraphics[scale=1]{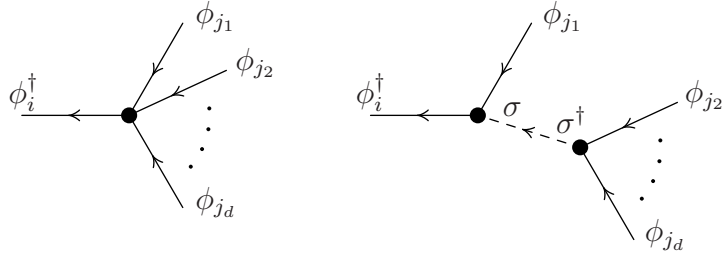}}
\put(0,1.2){$\phi^\dag_i$}
\put(1.4,1.8){$\phi_{j_1}$}
\put(1.7,1.45){$\phi_{j_2}$}
\put(1.4,0.4){$\phi_{j_d}$}
\put(2.6,1.2){$\phi^\dag_i$}
\put(4.0,1.8){$\phi_{j_1}$}
\put(5.063,1.206){$\phi_{j_2}$}
\put(4.763,0.156){$\phi_{j_d}$}
\put(3.7,1.1){$\sigma$}
\put(4.1,0.974){$\sigma^\dag$}
\end{picture}
  \caption[sigma]{Diagrammatic representation of the intermediate field method.}
  \label{fig-sigma}
\end{figure}
As stated in Section \ref{subsec-eur}, $\tilde
u=(u_1,\ldots,u_n,0,\ldots,0)$, and the coefficients $\tilde w$ are
\begin{itemize}
\item for $k=d-1$, set $\tilde w^{(d-1)}_{i,j,j_2...j_d}:=
  w^{(d-1)}_{i,j,j_2...j_d}$; and $\tilde w^{(d-1)}_{i\fois
    n+j,j_2...j_d}= w^{(d)}_{i,j,j_2...j_d}$ with $i,j,j_2,...j_n\leq
  n$.
\item for $k\in\{3,...,d-2\}$, set $\tilde w^{(k)}_{i,j,j_2...j_k}:=w^{(k)}_{i,j,j_2...j_k}$ with $i,j,j_2,...j_n\leq n$.
\item for $k=2$, set $\tilde w^{(2)}_{i,j,j_2}:=w^{(2)}_{i,j,j_2}=0$; and $\tilde w^{(2)}_{i,j,i\fois n+j}=1$ with $i,j,j_2\leq n$.
\end{itemize}
The other components of $\tilde w$
are set to $0$ by definition. Note that this corresponds to consider
the new polynomial map $\tilde F:\gC^{n(n+1)}\to\gC^{n(n+1)}$ given by
\begin{align*}
\tilde F_i(z)
&=
z_i-\sum_{k=3}^{d-1}\sum_{j_1,...,j_k=1}^nw^{(k)}_{i,j_1...j_k}z_{j_1}...z_{j_k}
-\sum_{j=1}^n z_j z_{i\fois n+j}
&&
\text{ for } 1\leq i\leq n,
\\
\tilde F_{i\fois n+j}(z)
&=
z_{i\fois
  n+j}-\sum_{j_2,...,j_d=1}^nw^{(d)}_{i,j,j_2,...j_k}z_{j_2}...z_{j_k}
&&
\text{ for } 1\leq i,j\leq n.
\end{align*}
To each tree $T$ in the theory of dimension $n$, and a choice of
incoming edge per each vertex of degree $d+1$, we can associate
canonically a tree $\tilde T$, constructed from these new Feynman
rules.  Propagators, leaves and vertices of coordination less than or
equal to $d$ in the tree $T$ are identically transposed in the tree
$\tilde T$, while a vertex in $T$ of coordination $d+1$ with one
outgoing edge $i$ and $d$ incoming edges of indices $j_1,\dots,j_d$ is
split into two vertices in $\tilde T$, connected by an edge of index
$n < i \leq n(n+1)$.  Edges of such indices are never adjacent on the
tree.  The precise construction is depicted in Figure~\ref{fig-sigma}.


Consider now the Feynman integral $\tilde\caA_i(\tilde T)(\tilde u)$,
in the model of dimension $n(n+1)$. 
For propagators, leaves and vertices of coordination less than or
equal to $d$ in the tree $T$, the contribution to $\tilde\caA_i(\tilde
T)(\tilde u)$ is the same as in $\caA_i(T)(u)$, because the summation
in $i,j_1,...,j_k$ of the vertices reduce to $\{1,...,n\}$, except for
$k=d-1$ where $i$ could a priori take value in
$\{n+1,...,n(n+1)\}$. However, the outgoing edge of this vertex is
either the external edge (with $i\in\{1,\dots,n\}$) or is adjacent to
another vertex (of coordination greater than or equal to four by
hypothesis), so we also have $i\in\{1,\dots,n\}$.

For any vertex of coordination $d+1$ in $T$, the above coefficients
$\tilde w^{(d-1)}_{i\fois n+j,j_2...j_d}$ and $\tilde
w^{(2)}_{i,j,i\fois n+j}$ have been chosen so that the contribution in
$\caA_i(T)(u)$ coincides with the one in $\tilde\caA_i(\tilde
T)(\tilde u)$. The only thing to check is that the index $\ell$
relating the two new vertices in $\tilde T$ is summed over only
$\ell\in\{n+1,\ldots, n(n+1)\}$. This is the case because $\tilde
w^{(2)}_{i,j,\ell}=0$ for $\ell\leq n$. Note also that the formal
factor $\theta^{d-1}$ for the vertex in $T$ corresponds to $\theta$
for the new vertex of coordination 3 and $\theta^{d-2}$ for the one of
coordination $d$ in $\tilde T$. Due to Feynman rules, only trees
$\tilde T$ obtained from a tree $T$ can contribute to the formal
inverse $\tilde G_i(\tilde u)$ of Theorem~\ref{thm-inv}.

One then concludes that 
\beqa
\tilde\caA_i(\tilde T)(\tilde u)=\caA_i(T)(u).
\eeqa
Due to Theorem \ref{thm-inv}, we have $\tilde G_i(\tilde u)=G_i(u)$
for $i\leq n$. So we can associate injectively to a polynomial
function $F\in \caP_{n,d}$ another function $\tilde
F=\Phi(F)\in\caP_{n(n+1),d-1}$ and we proved that $
F\in\caJ_{n,d}\Leftrightarrow \tilde F\in \caJ_{n(n+1),d-1;n}$.
 
The same process can be performed for graphs without external edges, which leads to the equality of partition functions in both QFT models, the one of dimension $n(n+1)$ and the one of dimension $n$: 
\beqa 
 \tilde Z(0,\tilde u)=Z(0,u).
\eeqa 
Moreover, for the one point correlation function with one external leg in the auxiliary field, equation \eqref{eq-invgraph} leads to the following expression:
\begin{equation*}
\tilde G_{i\fois n +j}(\tilde u)=\theta^{d-2}\sum_{j_2,\dots,j_d} w^{(d)}_{i,j,j_2\dots j_d} G_{j_2}(u)\dots G_{j_d}(u)=R^{-1}_{i\fois n+j}(0,\theta G(u)),
\end{equation*}
since $R_{i\fois n+j}(v,u)=v_{i\fois n+j}-\sum_{j_2,\dots,j_d} w^{(d)}_{i,j,j_2\dots j_d} u_{j_2}\dots u_{j_d}$ (see Definition \ref{def-jjlin}). By Proposition \ref{prop-partfunc}, we then obtain
\beqa
\det(J_{\tilde F})(\theta G(u),R^{-1}(0,\theta G(u)))=\det(J_{\tilde F})(\theta \tilde G(\tilde u))=\tilde Z(0,\tilde u)^{-1}=Z(0,u)^{-1}\nonumber\\
=\det(J_F)(\theta G(u)).\quad 
\eeqa
This proves the second part of the Theorem, namely $ F\in\caJ^{\rm lin}_{n,d}\Leftrightarrow \tilde F\in \caJ^{\rm lin}_{n(n+1),d-1;n}$.
\end{proof}

\section{Example}
\label{sec-example}

Let us illustrate the conjecture in low dimension $n=2$, $n'=1$ and arbitrary degree $d$ in this section. This won't be useful for the Jacobian conjecture of course, because it shows the case $n=1$, $d+1$, which is trivial. But it will give explicit computations involving the definitions introduced above.

We consider the polynomial given by
\begin{align*}
&F_1(z)=z_1-\sum_{k=0}^d a_{1,k}z_1^kz_2^{d-k}\\
&F_2(z)=z_2-\sum_{k=0}^d a_{2,k}z_1^kz_2^{d-k},
\end{align*}
where the complex coefficients $a_{i,k}$ are fixed. Then the Jacobian takes the form
\begin{multline}
\det(J_F)(z)=1-\sum_{k=0}^{d-1}(a_{1,k+1}(k+1)+a_{2,k}(d-k))z_1^kz_2^{d-1-k}\\
+\sum_{k,l=0}^d a_{1,k}a_{2,l}(d-k)l z_1^{k+l-1}z_2^{2d-k-l-1}\label{eq-ex-jac}
\end{multline}

In the standard case $n'=n=2$, by using \eqref{eq-ex-jac}, the equation $\det(J_F)=1$ leads to the conditions
\begin{itemize}
\item For any $k\in\{0,\dots,d-1\}$, $a_{1,k+1}(k+1)=a_{2,k}(d-k)$.
\item For any $m\geq 1$, $\sum_{k=0}^{\min(d,m)}a_{1,k}a_{2,m-k}d(2k-m)=0$.
\end{itemize}
If $F$ satisfies these conditions, $F$ lies in $\caJ^{\rm lin}_{2,d}$.
\medskip

Let us describe the polynomials $F$ that belong to $\caJ^{\rm lin}_{2,d;1}$ and compare with the above conditions. We will see that they are very different. Moreover, we will see that $\caJ^{\rm lin}_{2,d;1}=\caJ_{2,d;1}$.

For $n'=1$, we have to set
\begin{equation*}
R(z_2;z_1)=z_2-\sum_{k=0}^d a_{2,k}z_1^kz_2^{d-k}.
\end{equation*}
This expression has to be invertible as a polynomial in $z_2$ and for any parameter $z_1\in\gC$. In particular, the Jacobian of $R$ with respect to $z_2$ has to be constant, which implies $a_{2,k}=0$ for any $k<d-1$. So $R(z_2;z_1)=z_2-a_{2,d-1}z_1^{d-1}z_2-a_{2,d}z_1^d$. But the invertibility of $R$ for any $z_1$ also implies that $a_{2,d-1}=0$.

Eventually, we get $R(z_2;z_1)=z_2-a_{2,d}z_1^d$, and $R^{-1}(y_2;z_1)=y_2+a_{2,d}z_1^d$. Let us look at the condition $\det(J_F)(z_1,R^{-1}(0;z_1))=1$. Replacing $z_2$ by $R^{-1}(0;z_1)=a_{2,d}z_1^d$ in \eqref{eq-ex-jac}, we find the following expression
\begin{multline*}
\det(J_F)(z_1,R^{-1}(0;z_1))=1+\sum_{k=1}^d a_{1,k}a_{2,d}^{d-k}(k(d-1)-d^2)z_1^{(d-1)(d-1-k)}\\
+a_{1,0}a_{2,d}^dz_1^{(d-1)(d+1)}.
\end{multline*}
Then, the polynomial $F$ belongs to $\caJ^{\rm lin}_{2,d;1}$ if and only if $a_{1,d}=0$ and
\begin{equation*}
 \forall k\in\{0,\dots,d-1\}, \, a_{1,k}=0 \qquad {\text{or}} \qquad a_{2,d}=0.
\end{equation*}
We see indeed that these conditions are very different from the one of $\caJ^{\rm lin}_{2,d}$. Now, let us show that these polynomials $F\in\caJ^{\rm lin}_{2,d;1}$ are also in $\caJ_{2,d;1}$, so $(F^{-1})_1(z_1,0)$ is polynomial in $z_1$.

The first case of $\caJ^{\rm lin}_{2,d;1}$ corresponds to
\begin{equation*}
F_1(z)=z_1,\qquad F_2(z)=z_2-a_{2,d}z_1^d.
\end{equation*}
Here, the global inverse is
\begin{equation*}
F^{-1}_1(y)=y_1,\qquad F_2(y)=y_2+a_{2,d}y_1^d,
\end{equation*}
so the condition of $\caJ_{2,d;1}$ is trivially satisfied. The second case coincides with polynomials
\begin{equation*}
F_1(z)=z_1-\sum_{k=0}^{d-1}a_{1,k}z_1^kz_2^{d-k},\qquad F_2(z)=z_2.
\end{equation*}
The global inverse is not polynomial in $y_1,y_2$. However, by setting $u=(F^{-1})_1(y_1,0)$, we have the following equation $y_1=F_1(u,0)=u$, so
\begin{equation*}
(F^{-1})_1(y_1,0)=y_1
\end{equation*}
is polynomial in $y_1$, and $F\in \caJ_{2,d;1}$.

\section{Concluding remarks and perspectives}

We thus proved in this paper a reduction theorem to the quadratic case for the Jacobian conjecture, up to the addition of a new parameter $n'$. Moreover, we did this first by using formal algebraic methods and then using QFT methods. This idea of using intermediate field method represents an illustration of how QFT methods can be successfully used to prove ``purely'' mathematical results.

Recall here that the Jacobian Conjecture is proved in the quadratic case by Wang in \cite{Wang}. The immediate perspective thus appears to be the adaptation of Wang's proof to our particular case, where the parameter $n'$ plays a non-trivial role. An interesting approach for this may be the reformulation of Wang's proof in a QFT language, since we saw here that reduction results can be established in a natural way when using QFT techniques.

Let us end this paper by recalling that the Jacobian Conjecture 
is stably equivalent to the Dixmier Conjecture
for endomorphisms of the Weyl algebra.  It should be interesting to
revisit the Dixmier Conjecture from the perspective of Noncommutative
QFT (see \cite{Grosse:2004yu,GMRT, dGT, deGoursac:2012ki} and references within)
on the deformation quantization of the complex plane, which is an
extension of the Weyl algebra (see \cite{Omori:2000,Garay:2013gya}).
\bigskip

\noindent
{\bf{Acknowledgements:}}
\label{sec:ack}
A. Tanasa warmly thanks V. Rivasseau for bringing his attention to the Jacobian Conjecture in general and to the Abdesselam-Rivasseau combinatorial QFT model 
in particular; several discussions with J. Magnen at an early stage of this work are also warmly acknowledged. A. Tanasa was partially supported by the PN  09370102 and ANR JCJC CombPhysMat2Tens grants. A. de Goursac was partially supported by the ANR JCJC CombPhysMat2Tens grant and by the Belgian Interuniversity Attraction Pole (IAP) within the framework ``Dynamics, Geometry and Statistical Physics'' (DYGEST).

\nocite{*}
\bibliographystyle{alpha}
\bibliography{jacobiana}
\label{sec:biblio}

\bigskip

\noindent
Axel de Goursac\\
{\it\small Charg\'e de Recherche FNRS, Univ. Catholique de Louvain, Belgium}

\medskip
\noindent
Andrea Sportiello\\
{\it\small LIPN, Institut Galil\'ee, CNRS UMR 7030, \\
Universit\'e Paris 13, Sorbonne Paris Cit\'e,
Villetaneuse, France}

\medskip
\noindent
Adrian Tanasa\\
{\it\small LIPN, Institut Galil\'ee, CNRS UMR 7030, \\
Universit\'e Paris 13, Sorbonne Paris Cit\'e,
Villetaneuse, France\\
Horia Hulubei National Institute for Physics and Nuclear Engineering,
Magurele, Romania}

\end{document}